\documentclass[a4paper,reqno]{amsart}
\usepackage[utf8]{inputenc}
\usepackage{amssymb}
\usepackage{amsfonts}
\usepackage{amsmath}
\usepackage{amsthm}
\usepackage{indentfirst}

\newtheorem{theorem}{Theorem}[section]
\newtheorem{lemma}[theorem]{Lemma}

\newtheorem{remark}[theorem]{Remark}
\theoremstyle{remark}

\usepackage{graphicx} 
\newcommand{\ii}{\int_{\mathbb{R}^3}}
\newcommand{\lt}[1]{\Vert#1\Vert_{L^2}}
\newcommand{\hn}[1]{\Vert#1\Vert_{H^N}}
\newcommand{\hnn}[1]{\Vert#1\Vert_{H^{N+1}}}

\begin{document}

\title{NON-UNIQUENESS OF OBERBECK-BOUSSINESQ SYSTEM WITH GRAVITATIONAL FIELD}
\author{Xiaoran Liu}
\address{School of Mathematical Sciences, University of Science and Technology of China, Hefei 230026, Anhui, China}
\email{liuxiaoran@mail.ustc.edu.cn}

\begin{abstract}
    We consider the Oberbeck-Boussinesq system with gravitational force on the whole space. We prove the non-uniqueness of the system applying the unstable profile of Navier-Stokes equation.
\end{abstract}

\maketitle

\section{introduction}
\numberwithin{equation}{section}
We consider the Oberbeck-Boussinesq system on $\mathbb{R}^3$ with gravitational field
\begin{equation}
\begin{cases}
\partial_tu+u\cdot\nabla u+\nabla p=\Delta u+\theta\nabla G+f\\
\nabla\cdot u=0\\
\partial_t\theta+u\cdot\nabla\theta=\Delta\theta+h,
\end{cases}
\end{equation}
where $u(x,t)$ is the velocity field, $\theta(x,t)$ is the temperature, $p(x,t)$ is the pressure, $\nabla G=\frac{x}{|x|^3}$ is the gravitational force acting on the fluid and $f$ and $h$ are external force and heat source, respectively. The Oberbeck-Boussinesq equation is an approximation equation used to describe the motion of the incompressible fluid system driven by the thermal convection. Such a sysem was proposed by Oberbeck \cite{O} and Boussiness \cite{B}, and has been studied as a successful model for hydrodynamic stability. (1.1) can be derived as a singular low Mach number limit of the full Navier-Stokes-Fourier system, see \cite{BFO,FNS}. Saltzman \cite{BO} considered this problem using Fourier representation. Later, Lorenz \cite{L} evaluated the numerical solution and discovered the first strange attractor. Besides, there are also many works on generalized Oberbeck-Boussinesq system. For example, Rajagopal, Saccomandi and Vergori generalized this system for compressible fluid with pressure- and temperature-dependent material properties \cite{RSV}. Also Grandi and Passerini presented a new deriviation for gases \cite{GP}.\par
The well-posedness of Oberbeck-Boussinesq system on a bounded domain has been widely studied. In \cite{K} Komo proved the existence of Serrin type strong solution for small initial data. Also Abbatiello, Bul$\rm\acute a\check c$ek and Lear studied the weak solution of generalized Navier-Stokes-Fourier system \cite{AMD}. Moreover, Abbatiello and Feireisl proved the existence of weak solution with no-slip boundary condition \cite{AE}. In \cite{DG}, well-posedness of Oberbeck-Boussinesq with nonlinear diffusion was proved.\par
In some cases, $\nabla G$ is taken to be $(0,0,1)$. Brandolese and Schonbe \cite{LE} proved the existence of weak solution of such type equation, see also \cite{DP}. Global well-posedness was proved for small initial data, see \cite{GL,KP}. For results on more general versions of the Oberbeck-Boussinesq system, one could refer to \cite{DH,EEM}, to name a few. For an overview of the Oberbeck-Boussinesq system, see \cite{Z}.\par
On the other hand, for Oberbeck-Boussinesq on unbounded domains, we would also consider the case $G=\frac{1}{|x|}$ if the size of the object acting on fluid is negligible. In \cite{BK}, Brandolese and Karch proved the existence and uniqueness of solution for homogeneous initial data, as an extension of Jia and $\rm\check S$ver$\rm\acute a$k's work on Navier-Stokes \cite{HV}.\par
Despite of so many uniqueness results as mentioned, the converse is hardly studied. Since our non-uniqueness result is related to that of Navier-Stokes, we'd like to introduce some developments on non-uniqueness for hydrodynamic system.\par
Non-uniqueness of weak solutions is very important in turbulence theory. In \cite{S} Scheffer firstly gave a remarkable example of non-unique solution to 2 dimensional Euler equation, where the velocity lies in $L^2$ in space and time. Shnirelman then constructed non-unique solution in a different way \cite{S1,S2}. In 09' De Lellis and Szekelyzidi proved the non-uniqueness result for $v\in L^\infty_{t,x}$ for arbitrary dimension. More recently, Vishik \cite{V1,V2} considered the 2 dimensional Euler equation in vorticity form and proved the non-uniqueness for initial vorticity in $L^1\cap L^p$, which is sharp due to the classical work of Yudovich \cite{Y}. Based on Vishik's work, Albritton, Bru$\rm\acute e$ and Colombo proved the non-uniqueness result for 3 dimensional Navier-Stokes \cite{ABC}. It was also proved by Buckmaster amd Vicol, using the method of convex integration \cite{BV}. Furthermore, Bru$\rm\acute e$, Colombo and Kumar used a new convex integral scheme to prove the non-uniqueness of 2 dimensional Euler but without forcing term \cite{BCK}.\par

Our main theorem is as follows:
\begin{theorem}
    There exists $t_0>0,f,g\in L_t^1L_x^2((0,t_0)\times\mathbb{R}^3)$ such that (1.1) has infinitely many distributional solutions $(u,\theta)\in L_t^2L_x^2((0,t_0)\times\mathbb{R}^3)$ with initial data $(u_0,\theta_0)=(0,0)$.
\end{theorem}
The theorem tells that the turbulence could also cause the non-uniqueness of temperature distribution. One of the difficulties is that in the equation for temperature, the heat semigroup is not regular enough to control the convection term brought by the unstable profile $\bar u$. So our approach is to add this term into the generator and estimate the new semigroup $e^{\tau L}$. Due to similar reasons the temperature profile we choose is not steady, but smaller in time.\par

This paper is organized as follows. Section 2 introduces notations and known results about instability property of Navier-Stokes equation, which play an important role in our proof. In Section 3 we prove the main theorem by a contraction mapping argument.


\section{notations and known results}
It worth mentioning that (1.1) possesses the scaling property: if $(u,p,\theta,f,g)(x,t)$ is a solution, and set
$$
\begin{aligned}
    u_\lambda(x,t)=\lambda u(\lambda^2x,\lambda t),\theta_\lambda(x,t)=\lambda\theta(\lambda^2x,\lambda t),\\
    f_\lambda(x,t)=\lambda^3f(\lambda^2x,\lambda t),h_\lambda(x,t)=\lambda^3h(\lambda^2x,\lambda t),p_\lambda(x,t)=\lambda^2p(\lambda^2x,\lambda t),
\end{aligned}
$$
then $(u_\lambda,p_\lambda,\theta_\lambda,f_\lambda,h_\lambda)(x,t)$ is also a solution of (1.1). Note that the Navier-Stokes equations also have such property. We introduce the self-similar coordinates
$$
\xi=\frac{x}{\sqrt{t}},\tau=\log t.
$$
In the self similar coordinates we set
\begin{equation}
\begin{aligned}
    u(x,t)=\frac{1}{\sqrt{t}}U(\xi,\tau),\theta(x,t)=\frac{1}{\sqrt{t}}\Theta(\xi,\tau),\\
    f(x,t)=\frac{1}{t^\frac{3}{2}}F(\xi,\tau),h(x,t)=\frac{1}{t^\frac{3}{2}}H(\xi,\tau),p(x,t)=\frac{1}{t}P(\xi,\tau),
\end{aligned}
\end{equation}
then (1.1) becomes
\begin{equation}
\begin{aligned}
    \partial_\tau U-\frac{1}{2}(1+\xi\cdot\nabla)U-\Delta U+U\cdot\nabla U+\nabla P=\Theta\nabla\frac{1}{|\xi|}+F\\
    \nabla\cdot U=0\\
    \partial_\tau\Theta-\frac{1}{2}(1+\xi\cdot\nabla)\Theta-\Delta\Theta+U\cdot\nabla\Theta=H
\end{aligned}
\end{equation}
We say a solution is a self-similar solution, if it is stationary in self-similar coordinates. The study of self-similar solutions plays an important role in many problems. Firstly in \cite{JS,HV}, Jia, $\rm\check S$ver$\rm\acute a$k and Guillod found numerically the unstable profile. Moreover, in \cite{ABC} Albritton, Bru$\rm\acute e$ and Colombo proved this rigorously, which will be exploited in this paper. Later it was extended to related problems, like fractional Navier-Stokes \cite{LMZ} and MHD \cite{CCL}. In \cite{BK}, Brandoless and Karch studied the well-posedness of Oberbeck-Boussinesq for homogeneous initial data.\par
In the following we always use capital letters to represent the corresponding function in self-similar coordinates according to (2.1). For example, $u_p(x,t)=\frac{1}{\sqrt{t}}U(\frac{x}{\sqrt{t}},\log t)=e^{-\frac{\tau}{2}}U(\xi,\tau)$.\par
Consider the Navier-Stokes equation on $\mathbb{R}^3$
$$
\begin{cases}
\partial_tu+u\cdot\nabla u+\nabla p=\Delta u+f,\\
\nabla\cdot u=0.
\end{cases}
$$
In self-similar coordinates after applying the Leray projector, the first equation reads
$$
\partial_\tau U-\frac{1}{2}(1+\xi\cdot\nabla)U-\Delta U+U\cdot\nabla U+\nabla P=F.
$$
Therefore, if we linearize this equation at some profile $\bar U$, we get the linearized equation
$$
    \partial_\tau U-L_{ss}U=0
$$
where the linearized operator is
$$
L_{ss}U=\frac{1}{2}(1+\xi\cdot\nabla)U+\Delta U+\mathbb{P}(U\cdot\nabla\bar U+\bar U\cdot\nabla U).
$$
The instability of Navier-Stokes was studied by Albritton, Bru$\rm\acute e$, and Colombo in \cite{ABC}. They proved the existence of unstable profile $\bar U$ such that the linearized operator has an unstable eigenvalue:
\begin{lemma}[\cite{ABC}, Theorem 1.3]
    There exists a divergence-free vector field $\bar U\in C_c^\infty$ such that the linearized operator $L_{ss}:D(L_{ss})\subset L_\sigma^2\to L_\sigma^2$ has a maximally unstable eigenvalue $\lambda$ with smooth eigenfunction $\rho\in H^k\forall k$, that is
    $$
L_{ss}\rho=\lambda\rho
$$
and
$$
a=\Re\lambda=\sup_{z\in\sigma(L_{ss})}\Re z>0.
$$
Here $L_\sigma^2={u\in L^2:\nabla\cdot u=0}$.
\end{lemma}

\begin{remark}
It can be tracked from \cite{ABC} that by mulitplying the unstable profile by a large number, $a$ can be made sufficiently large. Without loss of generiosity we set $a>1$.
\end{remark}

Moreover, they also proved that the semigroup generated by $L_{ss}$ enjoy the parabolic regularity estimate.

\begin{lemma}[\cite{ABC}, Lemma 4.4]
    For all $\sigma_2\geq\sigma_1\geq0$, it holds
    $$
\Vert e^{\tau L_{ss}}U\Vert_{H^{\sigma_2}}\leq\frac{C(\sigma_1,\sigma_2,\delta)}{\tau^{(\sigma_2-\sigma_1)/2}}e^{(a+\delta)\tau}\Vert U\Vert_{H^{\sigma_1}}
    $$
    for any $U\in H^{\sigma_1}\cap L_\sigma^2$.
\end{lemma}

\section{proof of the main result}

\subsection{Strategy of the proof}
We firstly linearize (2.2) at a suitable profile, then decompose the solution as the combination of the profile, linear part and pertubation part. Letting $\bar U(\xi)$ be as in Lemma 2.1 and $\bar\Theta(\xi,\tau)$ a smooth compactly supported function of the form $e^{b\tau}\Theta(\xi)$ for some $b>0$, we solve $F$ and $H$ by plugging in $\bar U$ and $\bar\Theta$ into (1.1), that is
$$
\begin{aligned}
    F=&\partial_\tau\bar U-\frac{1}{2}(1+\xi\cdot\nabla)\bar U-\Delta \bar U+\mathbb{P}(\bar U\cdot\nabla\bar U-\bar\Theta\nabla\frac{1}{|\xi|})\\
    H=&\partial_\tau\bar\Theta-\frac{1}{2}(1+\xi\cdot\nabla)\bar\Theta-\Delta\bar\Theta+\bar U\cdot\nabla\bar\Theta.
\end{aligned}
$$
By direct calculation, $\Vert f(t)\Vert_{L^2}+\Vert h(t)\Vert_{L^2}\lesssim t^{-\frac{3}{4}}$, so $f,h\in L_t^1L_x^2$.\par
Nest, suppose $U=\bar U+U_l+U_p,\Theta=\bar\Theta+\Theta_l+\Theta_p$ is another solution with the same external force and heat source, where the lower index$ l$ stands for linear and $p$ stands for pertubation. Then $(U_l,\Theta_l)$ satisfies the linearized equation, which is
$$
\begin{cases}
    \partial\tau U_l-L_{ss}U_l=0\\
    \partial_\tau\Theta_l+\bar U\cdot\nabla\Theta_l-\Delta\Theta_l=0.
\end{cases}
$$
For simplicity, we apply Lemma 2.1 and pick
\begin{equation}
U_l=\Re(e^{\lambda\tau}\rho),\Theta_l=0.
\end{equation}
Then it remains to solve the equation satisfied by $(U_p,\Theta_p)$, that is
$$
\begin{cases}
    \partial_\tau U_p-L_{ss}U_p+\mathbb{P}(U_l\cdot\nabla U_l+U_l\cdot\nabla U_p+U_p\cdot\nabla U_l+U_p\cdot\nabla U_p-\Theta_p\nabla\frac{1}{|\xi|})=0\\
    \nabla\cdot U_p=0\\
    \partial_\tau\Theta_p-\Delta\Theta_p-\frac{1}{2}(1+\xi\cdot\nabla)\Theta_p+\bar U\cdot\nabla\Theta_p+U_l\cdot\nabla\bar\Theta+U_l\cdot\nabla\Theta_p+U_p\cdot\nabla\bar\Theta+U_p\cdot\nabla\Theta_p=0.
\end{cases}
$$
Using Duhamel formula, the first equation becomes
\begin{equation}
U_p(\tau)=\int_{-\infty}^\tau e^{(\tau-\tau')L_{ss}}\mathbb{P}(-U_l\cdot\nabla U_l-U_l\cdot\nabla U_p-U_p\cdot\nabla U_l-U_p\cdot\nabla U_p+\Theta_p\nabla\frac{1}{|\xi|})(\tau')d\tau'.
\end{equation}
For the third one, denote by $L=\Delta+\frac{1}{2}(1+\xi\cdot\nabla)+\bar U\cdot\nabla$. Then $L$ generates a strongly continuous semigroup, so we rewrite the equation as
\begin{equation}
\Theta_p(\tau)=\int_{_\infty}^\tau e^{(\tau-\tau')L}(-U_l\cdot\nabla\bar\Theta-U_l\cdot\nabla\Theta_p-U_p\cdot\nabla\bar\Theta-U_p\cdot\nabla\Theta_p)(\tau')d\tau'
\end{equation}
In the following section we will prove the existence of solution to the equation above.

\subsection{Estimate of $e^{\tau L}$}

In this section we prove that the semigroup generated by $L$ possesses the same regularity property as the heat semigroup.

\begin{lemma}
Suppose $k\geq m\geq0,\forall\Theta_0\in L^2,\tau>0$, we have
$$
\Vert e^{\tau L}\Theta_0\Vert_{H^k}\lesssim_{k,m}\max\{\tau^{-\frac{k-m}{2}},1\}\Vert\Theta_0\Vert_{H^m}
$$
\end{lemma}

\begin{proof}
We first prove the case when $k,m$ are integers. Then the general cases follow from interpolation.\par
Let $\Theta(\tau)=e^{\tau L}\Theta_0$. Then $\Theta(\tau)$ satisfies the following equation
$$
\begin{cases}
    \partial_\tau\Theta-\Delta \Theta-\frac{1}{2}(1+\xi\cdot\nabla)\Theta=-\bar U\cdot\nabla\Theta\\
    \Theta(0)=\Theta_0.
\end{cases}
$$
Then in natural coordinates $\theta(t)$ satisfies
\begin{equation}
\begin{cases}
    \partial_t\theta-\Delta\theta=-\bar u\cdot\nabla\theta\\
    \theta(1)=\theta_0.
\end{cases}
\end{equation}
Multiply by $\theta$ and integrate we find that
$$
\frac{1}{2}\frac{d}{dt}\ii \theta^2dx+\ii |\nabla\theta|^2dx=-\ii\theta\bar u\cdot\nabla\theta dx=0
$$
which means
$$
\lt{\Theta(\tau)}=t^{-\frac{1}{4}}\lt{\theta(t)}\leq t^{-\frac{1}{4}}\lt{\theta_0}=t^{-\frac{1}{4}}\lt{\Theta_0}<\lt{\Theta_0}.
$$
Let $\phi_k(t)=(t-1)^k\ii|D^k\theta(t)|^2dx$, then
$$
\begin{aligned}
    \dot\phi_k(t)=&k(t-1)^{k-1}\ii|D^k\theta|^2dx+(-1)^k(t-1)^k\ii\Delta^k\theta\cdot\partial_t\theta dx\\
    =&k(t-1)^{k-1}\ii|D^k\theta|^2dx+(-1)^k(t-1)^k\ii\Delta^k\theta(\Delta\theta-\nabla\cdot(\bar u\theta))dx\\
    \leq&k(t-1)^{k-1}\ii|D^k\theta|^2dx-(t-1)^k\ii|D^{k+1}\theta|^2dx\\
    &+(t-1)^k\ii|D^{k+1}\theta||D^k(\bar u\theta)|dx\\
    \leq&k(t-1)^{k-1}\ii|D^k\theta|^2dx+C(t-1)^k\ii|D^k(\bar u\theta)|^2dx.
\end{aligned}
$$
Since $\Vert D^k\bar u\Vert_{L^\infty}\sim t^\frac{k+1}{2}$, by Leibniz rule we get
$$
\dot\phi(t)\leq k(t-1)^{k-1}\Vert\theta\Vert_{H^k}^2~C(t-1)^kt^{-1}\Vert\theta\Vert_{H^k}^2\lesssim(t-1)^{k-1}\Vert\theta\Vert_{H^k}^2.
$$
Therefore
\begin{equation}
(t-1)^k\ii|D^k\theta(t)|^2dx=\phi(t)\lesssim(t-1)^{k-1}\int_1^t\Vert\theta(s)\Vert_{H^k}^2ds.
\end{equation}
Now we come to estimate the time integral of $H^k$ norm. Differentiate (3.4) for $k-1$ times we obtain that
$$
\partial_tD^{k-1}\theta-\Delta D^{k-1}\theta=-D^{k-1}(\bar u\cdot\nabla\theta).
$$
Multiply by $D^{k-1}\theta$ and integrate by parts, we obtain that
$$
\frac{1}{2}\frac{d}{dt}\ii|D^{k-1}\theta|^2dx+\ii|D^k\theta|^2dx=\ii D^k\theta D^{k-1}(\bar u\theta)dx.
$$
Since the highest order term $\ii D^{k-1}\theta\bar u\cdot\nabla D^{k-1}\theta dx$ vanishes, by Cauchy-Schwartz inequality and interpolation we have
$$
\frac{1}{2}\frac{d}{dt}\ii|D^{k-1}\theta|^2dx+\frac{1}{2}\ii|D^k\theta|^2dx\leq\Vert\theta\Vert_{H^{k-1}}^2.
$$
Integrate in time we see that
$$
\ii|D^k\theta(t)|^2dx+\int_1^t\ii|D^k\theta(s)|^2dxds\lesssim\ii|D^{k-1}\theta(1)|^2dx+\int_1^t\Vert\theta(s)\Vert_{H^{k-1}}^2ds.
$$
Therefore by induction we can show that
$$
\int_1^t\ii|D^k\theta(s)|^2dxds\lesssim\Vert\theta\Vert_{H^{k-1}}^2.
$$
Combining with (3.5), we proved that
$$
\ii|D^k\theta(t)^2dx\lesssim(t-1)^{-1}\Vert\theta(1)\Vert_{H^{k-1}}^2.
$$
Then it follows that
$$
\lt{D^k\Theta(\tau)}=\lt{D^k\theta(t)}t^{\frac{k}{2}-\frac{1}{4}}\lesssim(t-1)^{-\frac{1}{2}}t^{\frac{k}{2}-\frac{1}{4}}\Vert\theta_0\Vert_{H^{k-1}}\leq\tau^{-\frac{1}{2}}e^{(\frac{k}{2}-\frac{1}{4})\tau}\Vert\Theta_0\Vert_{H^{k-1}}.
$$
For $0<\tau<k$, we see that $\lt{D^k\Theta(\tau)}\lesssim\tau^{-\frac{1}{2}}\Vert\Theta_0\Vert_{H^{k-1}}$. For $\tau>k$, using the short time estimate we have
$$
\Vert\Theta(\tau)\Vert_{H^k}\lesssim\Vert\Theta(\tau-1)\Vert_{H^{k-1}}\lesssim\dots\lesssim\Vert\Theta(\tau-k)\Vert\leq\lt{\Theta_0}.
$$
Therefore we proved the lemma for $k-m=1$. The general cases follows from induction and interpolation.
\end{proof}

\subsection{estimates}
Let
$$
\begin{aligned}
    X=&\{U\in C_tH^N((-\infty,\tau_0)\times\mathbb{R}^3:\sup_{\tau<\tau_0}e^{-\beta\tau}\Vert U(\tau)\Vert_{H^N}<\infty\},\\
    Y=&\{\Theta\in C_tH^{N+1}((-\infty,\tau_0)\times\mathbb{R}^3:\sup_{\tau<\tau_0}e^{-\gamma\tau}\Vert\Theta(\tau)\Vert_{H^{N+1}}<\infty\},
\end{aligned}
$$
equipped with the norm
$$
\begin{aligned}
    \Vert U\Vert_X=&\sup_{\tau<\tau_0}e^{-\beta\tau}\Vert U(\tau)\Vert_{H^N},\\
    \Vert\Theta\Vert_Y=&\sup_{\tau<\tau_0}e^{-\gamma\tau}\Vert\Theta(\tau)\Vert_{H^{N+1}}.
\end{aligned}
$$
Here $\beta,\gamma>0$ are to be determined, $\frac{3}{2}<N<2$.\par
We start by estimating each term on the right hand side of (3.2). Using (3.1) and Lemma 2.3, with $H^N$ being an algebra, we get
$$
\begin{aligned}
\hn{e^{(\tau-\tau')L_{ss}}\mathbb{P}(U_l\cdot\nabla U_l)(\tau')}\leq&e^{2a\tau'}e^{(a+\delta)(\tau-\tau')}\hn{\rho\nabla\rho}\\
\lesssim&e^{2a\tau'}e^{(\tau-\tau')(a+\delta)}
\end{aligned}
$$

$$
\begin{aligned}
\hn{e^{(\tau-\tau')L_{ss}}\mathbb{P}(U_l\cdot\nabla U_p)(\tau')}\lesssim&e^{(a+\delta)(\tau-\tau')}(\tau-\tau')^{-\frac{1}{2}}\Vert\nabla(U_lU_p)(\tau')\Vert_{H^{N-1}}\\
\lesssim&e^{a\tau'}e^{(\tau-\tau')(a+\delta)}(\tau-\tau')^{-\frac{1}{2}}\hn{U_p(\tau')}\\
\leq&e^{a\tau'}e^{\beta\tau'}e^{(a+\delta)(\tau-\tau')}(\tau-\tau')^{-\frac{1}{2}}\Vert U_p\Vert_X
\end{aligned}
$$

$$
\begin{aligned}
\hn{e^{(\tau-\tau')L_{ss}}\mathbb{P}(U_p\cdot\nabla U_l)(\tau')}\lesssim&e^{(a+\delta)(\tau-\tau')}\hn{U_p\nabla U_l(\tau')}\\
\lesssim&e^{a\tau'}e^{(a+\delta)(\tau-\tau')}\hn{U_p(\tau')}\\
\lesssim&e^{a\tau'}e^{\beta\tau'}e^{(a+\delta)(\tau-\tau')}\Vert U_p\Vert_X
\end{aligned}
$$

$$
\begin{aligned}
\hn{e^{(\tau-\tau')L_{ss}}\mathbb{P}(U_p\cdot\nabla U_p)(\tau')}\lesssim&e^{(a+\delta)(\tau-\tau')}(\tau-\tau')^{-\frac{1}{2}}\Vert\nabla(U_pU_p)(\tau')\Vert_{H^{N-1}}\\
\lesssim&e^{(a+\delta)(\tau-\tau')}(\tau-\tau')^{-\frac{1}{2}}\hn{U_p(\tau')}^2\\
\leq&e^{2\beta\tau'}e^{(a+\delta)(\tau-\tau')}(\tau-\tau')^{-\frac{1}{2}}\Vert U_p\Vert_X^2
\end{aligned}
$$

To estimate the gravity term, we exploit that the Leray projector kills gradient fields. Therefore we have $\mathbb{P}(\Theta\nabla\frac{1}{|\xi|})=-\mathbb{P}(\frac{1}{|\xi|}\nabla\Theta)$. Then using Hardy inequality, we get

$$
\begin{aligned}
\hn{e^{(\tau-\tau')L_{ss}}\mathbb{P}(\Theta\nabla\frac{1}{|\xi|})(\tau')}\lesssim&e^{(a+\delta)(\tau-\tau')}(\tau-\tau')^{-\frac{N}{2}}\lt{\frac{1}{|\xi|}\nabla\Theta(\tau')}\\
\lesssim&e^{(a+\delta)(\tau-\tau')}(\tau-\tau')^{-\frac{N}{2}}\hn{\Theta_p(\tau')}\\
\lesssim&e^{(a+\delta)(\tau-\tau')}(\tau-\tau')^{-\frac{N}{2}}\Vert\Theta_p(\tau')\Vert_{H^2}\\
\leq&e^{\gamma\tau'}e^{(a+\delta)(\tau-\tau')}(\tau-\tau')^{-\frac{N}{2}}\Vert\Theta_p\Vert_Y
\end{aligned}
$$

Therefore, integrate the five inequalities above from $-\infty$ to $\tau$ we see that

\begin{equation}
    \hn{\int_{-\infty}^\tau e^{(\tau-\tau')L_{ss}}\mathbb{P}(U_l\cdot\nabla U_l)(\tau')d\tau'}\lesssim e^{2a\tau}
\end{equation}

\begin{equation}
    \hn{\int_{-\infty}^\tau e^{(\tau-\tau')L_{ss}}\mathbb{P}(U_l\cdot\nabla U_p)(\tau')d\tau'}\lesssim e^{(a+\beta)\tau}\Vert U_p\Vert_X
\end{equation}

\begin{equation}
    \hn{\int_{-\infty}^\tau e^{(\tau-\tau')L_{ss}}\mathbb{P}(U_p\cdot\nabla U_l)(\tau')d\tau'}\lesssim e^{(a+\beta)\tau}\Vert U_p\Vert_X
\end{equation}

\begin{equation}
    \hn{\int_{-\infty}^\tau e^{(\tau-\tau')L_{ss}}\mathbb{P}(U_p\cdot\nabla U_p)(\tau')d\tau'}\lesssim e^{2\beta\tau}\Vert U_p\Vert_X^2
\end{equation}

\begin{equation}
    \hn{\int_{-\infty}^\tau e^{(\tau-\tau')L_{ss}}\mathbb{P}(U\Theta_p\nabla\frac{1}{|\xi|})(\tau')d\tau'}\lesssim e^{\gamma\tau}\Vert\Theta_p\Vert_Y
\end{equation}

In this process we must set
\begin{equation}
    a>\delta,\beta>\delta,\gamma>a+\delta
\end{equation}
to make the integrals converge.\par

Now we estimate (3.3). Using Lemma 3.1,

$$
\begin{aligned}
\hnn{e^{(\tau-\tau')L}(U_l\cdot\nabla\bar\Theta)(\tau')}\lesssim&\max\{(\tau-\tau')^{-\frac{1}{2}},1\}\hn{U_l\cdot\nabla\bar\Theta(\tau')}\\
\lesssim&\max\{(\tau-\tau')^{-\frac{1}{2}},1\}\hn{U_l(\tau')}\hnn{\bar\Theta(\tau')}\\
\lesssim&\max\{(\tau-\tau')^{-\frac{1}{2}},1\}e^{(a+b)\tau'}
\end{aligned}
$$

$$
\begin{aligned}
\hnn{e^{(\tau-\tau')L}(U_l\cdot\nabla\Theta_p)(\tau')}\lesssim&\max\{(\tau-\tau')^{-\frac{1}{2}},1\}\hn{U_l\cdot\nabla\Theta_p(\tau')}\\
\lesssim&\max\{(\tau-\tau')^{-\frac{1}{2}},1\}\hn{U_l(\tau')}\hnn{\Theta_p(\tau')}\\
\lesssim&\max\{(\tau-\tau')^{-\frac{1}{2}},1\}e^{(a+\gamma)\tau'}\Vert\Theta_p\Vert_Y
\end{aligned}
$$

$$
\begin{aligned}
\hnn{e^{(\tau-\tau')L}(U_p\cdot\nabla\bar\Theta)(\tau')}\lesssim&\max\{(\tau-\tau')^{-\frac{1}{2}},1\}\hn{U_p\cdot\nabla\bar\Theta(\tau')}\\
\lesssim&\max\{(\tau-\tau')^{-\frac{1}{2}},1\}\hn{U_p(\tau')}\hnn{\bar\Theta(\tau')}\\
\lesssim&\max\{(\tau-\tau')^{-\frac{1}{2}},1\}e^{(\beta+\gamma)\tau'}
\end{aligned}
$$

Integrate in $\tau'$ we get

\begin{equation}
\hnn{\int_{-\infty}^\tau e^{(\tau-\tau')L}(U_l\cdot\nabla\bar\Theta)(\tau')d\tau'}\lesssim e^{(a+b)\tau}
\end{equation}

\begin{equation}
\hnn{\int_{-\infty}^\tau e^{(\tau-\tau')L}(U_l\cdot\nabla\Theta_p)(\tau')d\tau'}\lesssim e^{(a+\gamma)\tau}\Vert\Theta_p\Vert_Y
\end{equation}

\begin{equation}
\hnn{\int_{-\infty}^\tau e^{(\tau-\tau')L}(U_p\cdot\nabla\bar\Theta)(\tau')d\tau'}\lesssim e^{(\beta+b)\tau}\Vert U_p\Vert_X
\end{equation}

\begin{equation}
\hnn{\int_{-\infty}^\tau e^{(\tau-\tau')L}(U_p\cdot\nabla\Theta_p)(\tau')d\tau'}\lesssim e^{(\beta+\gamma)\tau}
\end{equation}

\subsection{contraction mapping}
Let $B=\{(U,\Theta):\Vert U\Vert_X\leq M,\Vert\Theta\Vert_Y\leq M,\nabla\cdot U=0\}$ where $0<M<1$ is to be determined, and set
$$
\begin{aligned}
    \Phi(U_p,\Theta_p)=&(\Phi_1(U_p,\Theta_p),\Phi_2(U_p,\Theta_p))\\
    \Phi_1(U_p,\Theta_p)(\tau)=&\int_{_\infty}^\tau e^{(\tau-\tau')L_{ss}}\mathbb{P}(-U_l\cdot\nabla U_l-U_l\cdot\nabla U_p-U_p\cdot\nabla U_l\\
    &-U_p\cdot\nabla U_p+\Theta_p\nabla\frac{1}{|\xi|})(\tau') d\tau'\\
    \Phi_2(U_p,\Theta_p)(\tau)=&\int_{_\infty}^\tau e^{(\tau-\tau')L}(-U_l\cdot\nabla\bar\Theta-U_l\cdot\nabla\Theta_p-U_p\cdot\nabla\bar\Theta-U_p\cdot\nabla\Theta_p)(\tau')d\tau'.
\end{aligned}
$$
We shall prove $\Phi$ is a contraction mapping on $B$. It follows from (3.6)-(3.10) and (3.12)-(3.15) that
$$
\begin{aligned}
    \Vert\Phi_1(U_p,\Theta_p)\Vert_X\leq&Ce^{(-\beta+\min\{a+\beta,2\beta,\gamma\})\tau_0}M+Ce^{(2a-\beta)\tau_0},\\
    \Vert\Phi_2(U_p,\Theta_p)\Vert_Y\leq&Ce^{(-\gamma+\min\{a+\gamma,\beta+b,\beta+\gamma\})\tau_0}M+Ce^{(a+b-\gamma)\tau_0}.
\end{aligned}
$$
Therefore, once we set
\begin{equation}
2a>\beta,\gamma>\beta,\beta+b>\gamma,a+b>\gamma,
\end{equation}
$\Phi$ maps $B$ to itself if $\tau_0$ is sufficiently small.\par
Now let $(U_1,\Theta_1),(U_2,\Theta_2)\in B$, and $(U,\Theta)=(U_1,\Theta_1)-(U_2,\Theta_2)$. We have
$$
\begin{aligned}
(\Phi_1(U_1,\Theta_1)-\Phi_1(U_2,\Theta_2))(\tau)=\int_{-\infty}^\tau e^{(\tau-\tau')L_{ss}}\mathbb{P}(-U_l\cdot\nabla U-U_\cdot\nabla U_l\\
-U_1\cdot\nabla U- U\cdot\nabla U_2+\Theta\nabla\frac{1}{|\xi|})(\tau')d\tau'
\end{aligned}
$$

$$
(\Phi_2(U_1,\Theta_1)-\Phi_2(U_2,\Theta_2))(\tau)=\int_{-\infty}^\tau e^{(\tau-\tau')L}(-U_l\cdot\nabla\Theta-U\cdot\nabla\bar\Theta-U_1\cdot\nabla\Theta-U\cdot\nabla\Theta_2)(\tau')d\tau'.
$$
Again using (3.6)-(3.10) and (3.12)-(3.15) we yield
$$
\Vert\Phi_1(U_1,\Theta_1)-\Phi_1(U_2,\Theta_2)\Vert_X\leq Ce^{(-\beta+\min\{a+\beta,2\beta,\gamma\})\tau_0}(\Vert U\Vert_X+\Vert\Theta\Vert_Y)
$$
$$
\Vert\Phi_2(U_1,\Theta_1)-\Phi_2(U_2,\Theta_2)\Vert_Y\leq Ce^{(-\gamma+\min\{a+\gamma,\beta+b,\beta+\gamma\})\tau_0}(\Vert U\Vert_X+\Vert\Theta\Vert_Y).
$$
Therefore $\Phi$ is a contraction mapping if $\tau'$ is sufficiently small. Hence there exists a unique $(U_p,\Theta_p)\in B$ that solves (3.2).\par
Finally, to avoid that $u_p$ cancels off $u_l$, in addition to (3.11) and (3.16) we further require
$$
\beta>a.
$$
\begin{remark}
    Multiplying the eigenfunction by different constants we get different $U_l$'s, say, $U_l^1$ and $U_l^2$. And we have proved that there exists $U_p^1,U_p^2,\theta_p^1,\theta_p^2$ such that $(\bar U+U_l^i+U_p^i,\bar\theta+\theta_p^i),i=1,2$ solves (1.1). Then their difference
    $$
    \hn{U_l^1+U_p^1-U_l^2-U_p^2}\geq C_1e^{a\tau}-C_2e^{\beta\tau}>0
    $$
    for small $\tau$, so we obtain infinite solutions by taking different $U_l$'s.
\end{remark}


\begin{thebibliography}{99}
\bibitem{AMD} A. Abbatiello, M. Bul$\rm\acute a\check c$ek, and D. Lear. On the existence of solutions to generalized Navier-Stokes-Fourier system with dissipative heating. \emph{Meccanica} (2024).
\bibitem{AE} A. Abbatiello and E. Feireisl. The Oberbeck-Boussinesq system with non-local boundary conditions. \emph{Quart. Appl. Math.} 81(2),297–306(2023).
\bibitem{ABC} D. Albritton, E. Bru$\rm\acute e$, and M. Colombo. Non-uniqueness of Leray solutions of the forced Navier-Stokes equations. \emph{Annals of Mathematics}, 196(1):415-455, 2022.
\bibitem{BFO} P. Bella, E. Feireisl and F. Oschmann. Rigorous derivation of the Oberbeck-Boussinesq approximation revealing unexpected term. \emph{Comm. Math. Phys.} 403(3):1245–1273, 2023.
\bibitem{BO} B. Boltzman. Finite Amplitude Free Convection as an Initial Value Problem-I. \emph{J.Atmos.Sci.}  19,329(1962).
\bibitem{B} J. Boussinesq. \emph{The$\rm\acute e$orie analytique de la chaleur} (Gauthier-Villars, 1903).
\bibitem{BK} L. Brandolese and G. Karch. Large self-similar solutions to Oberbeck-Boussinesq system with Newtonian gravitational field. \emph{arXiv}: 2311.01093v1.
\bibitem{LE} L. Brandolese and M.E. Schonbek. Large time decay and growth for solutions of a viscous Boussinesq system. \emph{Trans. Amer. Math. Soc.} 364, 10 (2012), 5057–5090.
\bibitem{BCK} E. Bru$\rm\acute e$, M. Colombo and A. Kumar. Flexibility of Two-Dimensional Euler Flows with Integrable Vorticity. \emph{arXiv}: 2408.07934.
\bibitem{BV} T. Buckmaster and V. Vicol. Nonuniqueness of weak solutions to the Navier-Stokes equation. \emph{Ann. of Math. (2)} 189(1):101–144, 2019. 
\bibitem{EEM} B. Climent-Ezquerra, E. Ortega-Torres, M. Rojas-Medar. On the regularity of weak solution for the Oberbeck-Boussinesq equations. \emph{arXiv}: 2402.15207v1.
\bibitem{DH} Raph$\rm\ddot a$el Danchin and Lingbing He, The Oberbeck-Boussinesq approximation in critical spaces. \emph{Asymptot. Anal.} 84(2013), no. 1-2, 61-102. MR 3134744.
\bibitem{DP} R. Danchin and M. Paicu. Existence and uniqueness results for the Boussinesq system with data in Lorentz spaces. \emph{Physica D} 237 (10-12) (2008), pp. 1444-1460.
\bibitem{DG} J.I. Diaz and G. Galiano. Existence and uniqueness of solutions of the Boussinesq system with nonlinear thermal diffusion. \emph{Topological Methods Nonlinear Anal.} 11(1): 59-82 (1998).
\bibitem{FNS} E. Feireisl and A. Novotn$\rm\acute y$. Singular limits in thermodynamics of viscous fluids. \emph{Advancesin Mathematical Fluid Mechanics.} Birkh$\rm\ddot a$user/Springer, Cham, 2017. Second edition.
\bibitem{GL} H. Gao and H. Liu. Global well-posedness and long time decay of the 3D Boussinesq equations.
\emph{J. Differ. Equ.} 263 (2017), pp. 8649-8665.
\bibitem{GP} D. Grandi, A. Passerini. On the oberbeck-Boussinesq approximation for gases. \emph{International Journal of Nonlinear Mechanics} 134: 103738, (2021).
\bibitem{HV} H. Jia and V. $\rm\check S$ve$\rm\acute r$ak. Local-in-space estimates near initial time for weak solutions of the Navier-Stokes equations and forward self-similar solutions. \emph{Invent. Math.} 196, 1(2014), 233-265.
\bibitem{JS} H. Jia and V. $\rm\check S$ve$\rm\acute r$ak. Are the incompressible 3d Navier-Stokes equations locally ill-posed in the natural energy space? \emph{J. Funct. Anal.} 268(12):3734-3766, 2015.
\bibitem{KP} G. Karch and N. Prioux. Self-similarity in viscous Boussinesq equations. \emph{Proc. Amer. Math. Soc.} 136, 3 (2008), 879-888.
\bibitem{K} C. Komo, Optimal initial value conditions for the existence of strong solutions of the Boussinesq equations. \emph{Ann Univ Ferrara} 60 (2014), 377-396.
\bibitem{LMZ} B. Lai, C. Miao, and X. Zheng. Forward self-similar solutions of the fractional Navier Stokes equations. \emph{Adv. Math.} 352:981–1043, 2019.
\bibitem{CCL} C.-C. Lai. Forward discretely self-similar solutions of the MHD equations and the viscoelastic Navier-Stokes equations with damping. \emph{J. Math. Fluid Mech.} 21(3):Paper No. 38, 28, 2019.
\bibitem{L} E.N. Lorenz. Deterministic Nonperiodic Flow. \emph{J.Atmos.Sci.} 20,130(1963).
\bibitem{O} A. Oberbeck. Ueber die W$\rm\ddot a$rmleitung der Fl$\rm\ddot u$ssigkeiten bei Ber$\rm\ddot u$cksichtigung der Str$\rm\ddot u$mungeninfolge von Temperaturdifferenzen. \emph{Ann. Phys. Chem.} 243(6), 271–292 (1879).
\bibitem{RSV} K.R. Rajagopal, G. Saccomandi, L. Vergori. On the approximation of isochoric motions of fluids under different flow conditions. \emph{Proceedings of the Royal Society} A 471: 20150159, (2015).
\bibitem{S} V. Scheffer. An inviscid flow with compact support in space-time.\emph{J. Geom. Anal.} 3 (1993), 343–401, 1993.
\bibitem{S1} A. Shnirelman. On the nonuniqueness of weak solution of the Euler equation.\emph{Comm. Pure Appl. Math.} 50, 1997.
\bibitem{S2} A. Shnirelman. Weak solutions with decreasing energy of incompressible Euler equations.\emph{Comm. Math. Phys.} 210, 541–603, 2000.
\bibitem{V1} M. Vishik. Instability and non-uniqueness in the Cauchy problem for the Euler equations of an ideal incompressible fluid. Part I. \emph{arXiv}: 1805.09426.
\bibitem{V2} M. Vishik. Instability and non-uniqueness in the Cauchy problem for the Euler equations of an ideal incompressible fluid. Part II. \emph{arXiv}: 1805.09440.
\bibitem{Y} V. I. Yudovich. Non-stationary flows of an ideal incompressible fluid. \emph{Z. Vychisl. Mat. i Mat. Fiz.} 3, 1032–1066 (Russian), 1963.
\bibitem{Z} R.K. Zeytounian. Joseph Boussinesq and his approximation: a contemporary view. \emph{Compt. Rendus Mec.} 331 (2003) 575-586.
\end{thebibliography}
\end{document}